\documentclass[11pt, letter]{amsart}
\usepackage{upgreek}

\usepackage{amsfonts, amsthm, amssymb, amsmath, stmaryrd}
\usepackage{mathrsfs,array}
\usepackage{eucal,fullpage,times,color,enumerate,accents}
\usepackage{url}

\newtheorem{innercustomthm}{Main Theorem}
\newenvironment{customthm}[1]
  {\renewcommand\theinnercustomthm{#1}\innercustomthm}
  {\endinnercustomthm}

\usepackage[backref]{hyperref}
\hypersetup{
  colorlinks   = true,          
  urlcolor     = blue,          
  linkcolor    = purple,          
  citecolor   = blue             
}

\DeclareMathAlphabet{\mathbbold}{U}{bbold}{m}{n}

\usepackage{color}
\usepackage{mathrsfs}
\usepackage{amssymb}
\usepackage{tikz}
\usepackage{tikz-cd}
\usepackage{xypic}
\usepackage{bm}
\usepackage{enumerate}
\usetikzlibrary{calc}
\usetikzlibrary{fadings}

\title{{\larger\larger D}egenerations of toric varieties over valuation rings}
\author{{\larger T}{\smaller yler}\ \ {\larger F}{\smaller oster}\ \ \&\ \ {\larger D}{\smaller hruv}\ \ {\larger R}{\smaller anganathan}}
\date{\today}

\address{{\bf Tyler Foster}\newline Department of Mathematics, University of Michigan\\ 2074 East Hall 530 Church Street Ann Arbor, MI  48109-1043}
\email{tyfoster@umich.edu}

\address{{\bf Dhruv Ranganathan}\newline Department of Mathematics, Yale University\\ 10 Hillhouse Avenue, New Haven, CT 06511}
\email{dhruv.ranganathan@yale.edu}

\newtheorem{theorem}{Theorem}[subsection]

\newtheorem{lemma}[theorem]{Lemma}
\newtheorem{proposition}[theorem]{Proposition}
\newtheorem{definition}[theorem]{Definition}

\newtheorem{example}[theorem]{Example}
\newtheorem{quasi-theorem}[theorem]{Quasi-Theorem}
\newtheorem{blank remark}[theorem]{}

\newtheorem{rem1}[theorem]{Remark}
\newenvironment{remark}{\begin{rem1}\em}{\end{rem1}}

\newtheorem{not1}[theorem]{Notation}

\newcommand{\CC} {{\mathbb C}}

\newcommand{\PP}{\mathbb{P}}         
		
\newcommand{\RR} {{\mathbb R}}		
\newcommand{\ZZ} {{\mathbb Z}}

\newcommand{\Hom}{\operatorname{Hom}}

\DeclareMathOperator{\spec}{Spec}

\DeclareMathOperator{\red}{{red}}

\def\fH{\mathfrak{H}}

\def\fp{\mathfrak{p}}

\def\trop{\mathrm{trop}}

\newcommand{\Spec}{\operatorname{Spec}}

\def\lra{\longrightarrow}

\newcommand{\mono}{\!\xymatrix{{}\ar@{^{(}->}[r]&{}}\!}

\newcommand{\tinyge}{\mbox{{\smaller\smaller\smaller\smaller\smaller $\ge$}}}

\newcommand{\tyler}[1]{{\color{red}\bf [#1\ \ \textemdash\ \ Tyler]}}

\newif\ifshow
\showfalse

\newsavebox{\foobox}
\newcommand{\slantbox}[2][.5]
  {%
    \mbox
      {%
        \sbox{\foobox}{#2}%
        \hskip\wd\foobox
        \pdfsave
        \pdfsetmatrix{1 0 #1 1}%
        \llap{\usebox{\foobox}}%
        \pdfrestore
      }%
  }
  
\subjclass[2010]{14M25 and 14T05}

\begin{document}
\pagestyle{plain}
\maketitle

\vspace{-0.25in}

\begin{abstract}
We develop a theory of multi-stage degenerations of toric varieties over finite rank valuation rings, extending the Mumford--Gubler theory in rank one. Such degenerations are constructed from fan-like structures over totally ordered abelian groups of finite rank. Our main theorem describes the geometry of successive special fibers in the degeneration in terms of the polyhedral geometry of a system of recession complexes associated to the fan.
\end{abstract}

\setcounter{section}{0}
\setcounter{subsection}{0}

\section{Introduction}

In~\cite{KKMSD,Mum72}, Mumford describes how a rational polyhedral complex in a vector space gives rise to a degeneration of a toric variety over a discrete valuation ring. In~\cite{Gub13}, Gubler extends this theory to degenerations of toric varieties over arbitrary rank-$1$ valuation rings and Gubler and Soto~\cite{GS13} use these results to classify toric schemes over rank-$1$ valuation rings. These degenerations are a crucial ingredient in tropical geometry and have, for instance, been applied to study the enumerative geometry of toric varieties~\cite{NS06}. Recently, Payne and the first author use Gubler models to give a new description of the Huber analytification of a variety, as an inverse limit of adic tropicalizations~\cite{F15,FP15}.

The purpose of this article is to extend the theory of toric degenerations of toric varieties to valuation rings that have rank greater than one. Such multi-stage degenerations will play an important role in the theory of Hahn analytifications and higher rank tropicalizations being developed by the authors, as recently introduced in~\cite{FR1}.
	
\subsection{Rank-$1$ degenerations of toric varieties}

	Let $K$ be a field complete with respect to a nontrivial non-Archimedean valuation $v:K^{\times}\to\RR$. Let $R$ and $\widetilde K$ denote the valuation ring and residue field, respectively. The central construction of Gubler's theory associates to each {\em complete $\Gamma$-admissible fan} $\Sigma$ in $N_{\RR}\times\RR_{\ge0}$ (see \cite[Sections 6 and 7]{Gub13}) an $R$-scheme $\mathscr{Y}\!(\Sigma)$. One of the central results of \cite{Gub13} describes the geometry of $\mathscr{Y}\!(\Sigma)$ in terms of that of $\Sigma$:

\begin{theorem}\label{theorem: Gubler's Theorem}
{\bf (Gubler \cite{Gub13}).} For each complete {\em $\Gamma$-admissible fan $\Sigma$} in $N_{\RR}\times\RR_{\ge0}$, the $R$-scheme $\mathscr{Y}\!(\Sigma)$ is flat and proper over $R$, and satisfies the following:
	\begin{itemize}
	\item[{\bf (i)}]
	The reduced special fiber $\mathscr{Y}\!(\Sigma)^{\red}_{\widetilde{K}}$ is a collection of proper toric $\widetilde{K}$-varieties glued equivariantly along torus-invariant strata.
	\item[{\bf (ii)}]\vskip .2cm
	The irreducible components of $\mathscr{Y}\!(\Sigma)^{\red}_{\widetilde{K}}$ are in natural bijection with the vertices of the $\Gamma$-rational polyhedral complex $\Sigma\cap\big(N_{\RR}\!\times\!\{1\}\big)$ inside $N_{\RR}\!\times\!\{1\}\cong N_{\RR}$. The reduced irreducible component associated to a vertex is equivariantly isomorphic to the toric variety associated to the star of the vertex. 
	\item[{\bf (iii)}]\vskip .2cm
	The generic fiber $\mathscr{Y}\!(\Sigma)_{K}$ is naturally isomorphic to the proper toric $K$-variety associated to the fan $\Sigma\cap\big(N_{\RR}\!\times\!\{0\}\big)$ inside $N_{\RR}\!\times\!\{0\}\cong N_{\RR}$.
	\end{itemize}
If $K$ is algebraically closed, then the special fiber is reduced. 	
\end{theorem}
		
	The central result of the present article is an extension of this theorem to degenerations of toric varieties over arbitrary finite rank valuation rings, as we now explain.

\subsection{Brief reminder on finite rank valuation rings}\label{subsection: finite rank valuation rings}
	In order to state our main result, we briefly review the basic structure of finite rank valuation rings. For each nonnegative integer $k$, let $\RR^{(k)}$ denote the totally ordered abelian group given by $\RR^{k}$ with its lexicographic order. Let $K$ be a field equipped with a valuation $v:K^{\times}\to\RR^{(k)}$, and let $R$ denote the valuation ring $R:=\{0\}\cup\{a\in K^{\times}:v(a)\ge0\}$. The value group $\Gamma:=v(K^{\times})$ in $\RR^{(k)}$ inherits the structure of a totally ordered abelian group from the ambient total ordering on $\RR^{(k)}$. Recall that $n=\mathrm{rank}_{\ \!}\Gamma$ is the length of $\Gamma$'s maximal tower convex subgroups (\S\ref{subsection: Hahn embeddings})
	\begin{equation}\label{equation: convex subgroups in value group}
	\{0\}=\Delta_{0}\ \subset\ \Delta_{1}\ \subset\ \cdots\ \subset\ \Delta_{n}=\Gamma.
	\end{equation}
Each convex subgroup $\Delta_{i}\subset\Gamma$ determines a prime ideal $\mathfrak{p}_{i}:=\{a\in K:\ \!\forall\delta\in\Delta_{i},\ \!v(a)>\delta\}$ in $R$. Each $\Delta_{i}$ also determines a new valuation $v_{i}:K^{\times}\xrightarrow{v}\Gamma\twoheadrightarrow\Gamma/\Delta_{i}$, and the corresponding valuation rings $R_{i}:=\{a\in K:v_{i}(a)\ge0\}$ form an ascending tower
	$$
	R=R_{0}\subset R_{1}\subset\cdots\subset R_{n}=K.
	$$
For each $1\le i\le n$, the ideal $\mathfrak{p}_{i}$ in $R$ is in fact the unique maximal ideal inside $R_{i}$. The {\em $i^{\mathrm{th}}$ intermediate residue field of $K$} is the quotient
	$$
	\widetilde{K}_{i}\ :=\ R_{i}/\mathfrak{p}_{i}.
	$$
In particular, $\widetilde{K}_{n}\cong K$. For further details concerning valued fields of higher rank, see \cite{EP05}.
	
\subsection{Gubler models over finite rank valuation rings}
	The decreasing, length-$1$ tower $\RR_{\ge0}\supset\{0\}$ plays a central role in Gubler's theory of integral models. As we explain in Section \ref{subsection: flag of endomorphisms} below, our choice of order-preserving embedding $\Gamma\hookrightarrow \RR^{(k)}$ determines a decreasing, length $n=\mathrm{rank}_{\ \!}\Gamma$ tower of definable subsets in $\RR^{(k)}$:
	\begin{equation}\label{equation: subtower}
	\mathcal{E}=\mathcal{E}_{0}\ \supset\ \mathcal{E}_{1}\ \supset\ \cdots\ \supset\ \mathcal{E}_{n}.
	\end{equation}
This tower (\ref{equation: subtower}) generalizes the length-$1$ tower $\RR_{\ge0}\supset\{0\}$.
	
	Fix a $\ZZ$-lattice $M$. In Section~\ref{section: polyhedral geometry over Hahn embeddings}, we introduce {\em complete $\Gamma$-admissible fans} $\Sigma$ inside the product $\Hom_{\ZZ}(M,\RR^{k})\!\times\!\mathcal{E}$, and we use the tower (\ref{equation: subtower}) to produce a collection of {\em recession complexes} $\mathrm{rec}_{i}(\Sigma)\subset \Hom_{\ZZ}(M,\RR^{k})$ associated to $\Sigma$, for each $0\le i\le n$. In Section~\ref{section: models associated to polyhedral complexes}, we describe how to construct an $R$-scheme $\mathscr{Y}\!(\Sigma)$ from $\Sigma$, and our main result, which we prove throughout Section~\ref{section: models associated to polyhedral complexes}, is the following theorem.

\begin{customthm}{\!}\label{thm: tropical-theorem}
	For each complete $\Gamma$-admissible fan $\Sigma$ inside $\Hom_{\ZZ}(M,\RR^{(k)})\!\times\!\mathcal{E}$, the $R$-scheme $\mathscr{Y}\!(\Sigma)$ is flat and proper over $R$, and satisfies the following properties:
	\begin{itemize}
	\item[{\bf (i)}]
	For each $0\le i\le n$, the reduced intermediate fiber $\mathscr{Y}\!(\Sigma)^{\red}_{\widetilde{K}_{i}}\cong\big(\mathscr{Y}\!(\Sigma)\!\otimes_{R}\! R_{i}\big)^{\red}_{\widetilde{K}_{i}}$ is a collection of toric $\widetilde{K}_{i}$-varieties glued equivariantly along their torus-invariant strata.
	\item[{\bf (ii)}]\vskip .2cm
	The irreducible components of $\mathscr{Y}\!(\Sigma)^{\red}_{\widetilde{K}_{i}}$ are in natural bijection with the vertices of $\mathrm{rec}_{i}(\Sigma)$. The reduced irreducible component corresponding to a vertex is equivariantly isomorphic to the toric variety of the star of that vertex. 
	\item[{\bf (iii)}]\vskip .2cm
	The generic fiber $\mathscr{Y}\!(\Sigma)_{K}$ is the toric $K$-variety associated to $\mathrm{rec}_{n}(\Sigma)$.
	\end{itemize}
If $K$ is algebraically closed, then every intermediate fiber is reduced. 	
\end{customthm}

\begin{remark}
	The primary accomplishment of the present paper is to provide a polyhedral framework for $\Gamma$-admissible fans and their recession complexes in the context of higher rank toric degenerations. Once the machinery is in place, the proof of the Main Theorem is similar to the proof of Theorem~\ref{theorem: Gubler's Theorem}.
	
	One important source of higher rank toric degenerations is higher rank tropicalizations. Let $K$ be a valued field as above, with valuation $K^\times\twoheadrightarrow \Gamma\hookrightarrow \RR^{(k)}$. Let $X$ be a subvariety of a torus $T$ over $K$. In~\cite{FR1}, building on work of Aroca~\cite{Aroca10}, Banerjee~\cite{Ban13} and Nisse--Sottile~\cite{NS11}, we define and study the Hahn tropicalization $\trop(X)$ of $X$, which is a subset of $\Hom_\ZZ(M,\RR^{(k)})$. The set $\trop(X)$ admits the structure of a polyhedral complex over $\RR^{(k)}$, as defined in Section \ref{subsection: rational polyhedra} below. Using the constructions in the present paper, one can employ the Hahn tropicalization of $X$ to produce an equivariant compactification of $T$ to a toric variety, along with a multistage degeneration of this toric variety, such that the compactification of $X$ in the degeneration intersects components in intermediate special fibers properly. The rank-$0$ version of this was first studied by Tevelev~\cite{Tev07} and the rank-$1$ version was studied by Luxton--Qu~\cite{LQ11} and Gubler~\cite{Gub13}.
\end{remark}

\subsection{Sumihiro's Theorem in Higher Rank} We conclude with a brief discussion of how one might refine the results of this paper. In the theory of normal toric varieties over fields, one has, in addition to a construction of varieties from fans, a classification that all toric varieties arise by this construction. See, for instance~\cite{Ful93}. Working in Mumford's setting of degenerations over discrete valuation rings, one can use the fan that toric varieties are canonically defined over $\ZZ$ to obtain a similar classification result~\cite{KKMSD}. Over more general rank-$1$ valuation rings, the result was only recently proved by Gubler and Soto~\cite{GS13}, building on earlier work of Gubler~\cite{Gub13}. It would be interesting to have such a classification result over higher rank as well, to complete the picture. 

The classification theorem is proved by first proving a combinatorial classification for affine toric varieties and then using Sumihiro's theorem to pass to the general case. The latter states that any point in a normal toric variety is contained in an invariant affine open. The techniques of Gubler and Soto rely on an approximation technique which eventually allows one to reduce to the Noetherian case, and this strategy does not seem immediately applicable in higher rank. Indeed, even a valuation ring whose value group is $\ZZ^{(k)}$ is not Noetherian. We leave Sumihiro's theorem and the corresponding classification result as avenues for future investigation. 
	
\subsection*{Acknowledgements} We are especially grateful to Sam Payne for guidance and encouragement throughout the project. We have benefited from conversations with friends and colleagues, including Dan Abramovich, Matt Baker, Dan Corey, Walter Gubler, Max Hully, and Jeremy Usatine. T.F. was partially supported by NSF RTG grant DMS-09343832. D.R. was supported by NSF CAREER DMS-1149054 (PI: Sam Payne) and acknowledges ideal working conditions at Brown University during the spring and summer terms in 2015. We thank the anonymous referee for a number of very helpful comments that improved the paper. 

\section{Structure of value groups}\label{section: structure of value groups}

\subsection{Hahn embeddings}\label{subsection: Hahn embeddings}
Fix a totally ordered abelian group $\Gamma$.
	
	For any $\gamma\in \Gamma$, either $\gamma$ or $-\gamma$ is greater than $0$. Let $|\gamma|$ denote the larger of the two elements $\gamma$ and $-\gamma$. An element $\gamma'\in\Gamma$ is \textit{infinitely larger} than $\gamma$ if $m|\gamma|<\gamma'$ for every positive integer $m$. Elements $\gamma$ and $\gamma'$ are \textit{Archimedean equivalent} if neither $\gamma$ nor $\gamma'$ is infinitely larger than the other, and this defines an equivalence relation on $\Gamma$.

	A subgroup $\Delta$ of $\Gamma$ is called \textit{convex} if for each element $\delta\in \Delta$, any element $\gamma\in \Gamma$ satisfying $0\leq \gamma \leq \delta$ belongs to $\Delta$. Each convex subgroup is uniquely expressible as a union of Archimedean equivalence classes. The collection of nontrivial convex subgroups of $\Gamma$ is totally ordered by containment, and the {\em rank} of $\Gamma$ is the order type of the set of nonempty convex subgroups $\Gamma$. By a seminal theorem of Hahn~\cite{Hahn}, every finite rank totally ordered abelian group admits an order-preserving embedding
	\begin{equation}\label{equation: Hahn embedding}
	\Gamma\mono\RR^{(k)}
	\end{equation}
for some integer $k\ge0$, where $\RR^{(k)}$ denotes the additive group $\RR^{k}$ equipped with its lexicographic order. We refer to (\ref{equation: Hahn embedding}) as a {\em Hahn embedding}.
	
	For each $j\le k$, there is a unique order-preserving inclusion $\RR^{(j)}\hookrightarrow\RR^{(k)}$ whose image is a convex subgroup of $\RR^{(k)}$, namely the inclusion taking $(r_{1},\dots,r_{j})\mapsto(0,\dots,0,r_{1},\dots,r_{j})$. The steps in the resulting tower
	\begin{equation}\label{equation: convex subgroups of R^(k)}
	\{0\}\hookrightarrow\RR\hookrightarrow\RR^{(2)}\hookrightarrow\cdots\hookrightarrow\RR^{(k)}
	\end{equation}
are in bijection with the convex subgroups of $\RR^{(k)}$, and $\mathrm{rank}_{\ \!}\RR^{(k)}=k$.

\begin{lemma}\label{lemma: all embeddings are strict}
Fix a Hahn embedding (\ref{equation: Hahn embedding}), and let $\Delta$ be a convex subgroup of $\Gamma$. Then there exists a $0\leq j \leq k$ such that $\Delta=\RR^{(j)}\cap \Gamma$, where $\RR^{(j)}\subset \RR^{(k)}$ is identified with the subgroup whose nonzero entries lie in the last $j$ coordinates. 
\end{lemma}

\begin{proof}
Write $\Delta$ as a disjoint union $\Delta\ =\ [0]\sqcup[\gamma_{1}]\sqcup\cdots\sqcup[\gamma_{j}]$ of Archimedean equivalence classes in $\Gamma$, such that $0<\gamma_{1}<\cdots<\gamma_{i}$. Then $\Delta$ is the union of all elements $\gamma\in\Gamma$ {\em not} infinitely larger than $\gamma_{i}$. Similarly, the union of all elements not infinitely larger than $\gamma_{i}$ inside $\RR^{(k)}$ is a convex subgroup $\RR^{(j)}\hookrightarrow\RR^{(k)}$. For each $\gamma\in\Gamma$, the relation ``$\gamma$ is not infinitely larger than $\gamma_{i}$" holds in $\Gamma$ if and only if it holds in $\RR^{(k)}$. Hence $\RR^{(j)}\cap\Gamma=\Delta$.
\end{proof}

\subsection{Flags of endomorphisms inside $\pmb{\RR^{(k)}}$}\label{subsection: flag of endomorphisms}
	Fix a Hahn embedding (\ref{equation: Hahn embedding}) and let (\ref{equation: convex subgroups in value group}) denote the maximal tower of convex subgroups in $\Gamma$ . Coordinatewise multiplication by any vector $\underline{r}=(r_{1},\dots,r_{k})\in\RR^{k}$ defines a homomorphism of abelian groups
	$$
	\varphi_{\underline{r}}:\RR^{(k)}\lra\RR^{(k)}
	\ \ \ \ \ \ \mbox{taking}\ \ \ \ \ 
	(s_{1},\dots,s_{k})\mapsto(r_{1}s_{1},\dots,r_{k}s_{k}).
	$$
Yet not all vectors $\underline{r}$ in $\RR^{k}$ define an order-preserving homomorphism $\varphi_{\underline{r}}$. We define
	\begin{equation}\label{equation: two descriptions of the flag}
	\begin{array}{rcl}
	\mathcal{E}
	&
	:=
	&
	\big\{\ \!\underline{r}\in\RR^{k}\ \!:\ \!\varphi_{\underline{r}}\mbox{ is order-preserving}\ \!\big\}
	\\[6pt]
	&
	=
	&
	\big\{\ \!\underline{r}=(r_{1},\dots,r_{k})\in\RR^{k}\ \!:r_1\geq 0\ \text{and}\ \forall 0\leq i<k,\ r_i=0\ \Rightarrow\ r_{i+1}\geq 0\big\},\!\!\!\!\!\!\!\!\!\!\!\!\!\!
	\end{array}
	\end{equation}
and we equip $\mathcal{E}$ with the subspace topology it inherits from the lexicographic order topology on $\RR^{k}$, i.e., with its subspace topology in $\RR^{(k)}$. The $\mathcal{E}$-action on $\RR^{(k)}$ restricts to a pairing
	\begin{equation}\label{equation: gamma pairing}
	\begin{array}{rcl}
	\Gamma\times\mathcal{E}\!\!
	&
	\lra
	&
	\RR^{(k)}
	\\
	(\gamma,\underline{r})
	&
	\longmapsto
	&
	\!\!\varphi_{\underline{r}}(\gamma).
	\end{array}
	\end{equation}
Composition $\varphi_{\underline{r}}\circ\varphi_{\underline{s}}$ gives $\mathcal{E}$ the structure of a commutative monoid, written multiplicatively, and pointwise addition extends this to the structure of a commutative semiring with multiplicative identity $(1,\dots,1)$ and additive identity $(0,\dots,0)$.
	
By Lemma \ref{lemma: all embeddings are strict}, there is a strictly increasing sequence of integers $0=j_{0}<\cdots<j_{n}\le k$ such that $\RR^{(j_{i})}$ is the convex hull of $\Delta_{i}$ inside $\RR^{(k)}$, and such that $\Delta_{i}=\Gamma\cap\RR^{(j_{i})}$. Define
	$$
	\mathcal{E}_{i}
	\ :=\ 
	\mathcal{E}\ \cap\ \RR^{(k-j_{i})}
	\ \ \ \ \mbox{inside}\ \ \ \ 
	\RR^{(k)},
	$$
and note that the $\mathcal{E}_{i}$'s fit into a strictly decreasing tower (\ref{equation: subtower}). The second description of $\mathcal{E}$ appearing in (\ref{equation: two descriptions of the flag}) shows that each $\mathcal{E}_{i}$ is a definable subset of $\RR^{(k)}$ (see \cite[Section 2.5]{FR1} for a discussion of definablility). Note also that for $i\ge 1$, the set $\mathcal{E}_{i}$ depends on $\Gamma$ and its embedding $\Gamma\hookrightarrow\RR^{(k)}$, whereas $\mathcal{E}_{0}=\mathcal{E}$ regardless of our choice of $\Gamma$ and its Hahn embedding.
	
	The semiring structure on $\mathcal{E}$ restricts to a semiring structure (without unit) on each $\mathcal{E}_{i}$, and $\mathcal{E}_{i}$ acts on $\RR^{(k)}$ by order-preserving endomorphisms that map $\RR^{(k)}$ into the convex subgroup $\RR^{(k-j_{i})}\subset\RR^{(k)}$. Henceforth, we denote elements of $\mathcal{E}_{i}$ with greek letters $\varphi$, $\psi$, etc., undecorated by their corresponding elements $\underline{r}\in\RR^{k}$. We let $\varepsilon_{i}\in\mathcal{E}_{i}$ denote the order-preserving endomorphism $\varepsilon_{i}:\RR^{(k)}\to\RR^{(k)}$ given by
	\begin{equation}\label{equation: explicit projection description}
	\varepsilon_{i}(r_{1},\dots,r_{k-j_{i}},\dots,r_{k})
	\ =\ 
	(r_{1},\dots,r_{k-j_{i}},0,\dots,0).
	\end{equation}
The next lemma follows easily from Lemma \ref{lemma: all embeddings are strict}.

\begin{lemma}\label{lem: embeddedquotients}
For each $0\le i\le k$, there exists a unique order-preserving, injective homomorphism $\Gamma/\Delta_{i}\hookrightarrow\RR^{(k)}$ that makes the following diagram commute
	\begin{equation}\label{equation: important commutative square}
	\begin{aligned}
	\begin{xy}
	(0,0)*+{\ \Gamma\ }="1";
	(15,0)*+{\RR^{(k)}\!\!\!\!\!}="2";
	(0,-12)*+{\ \Gamma/\Delta_{i}\ }="3";
	(15,-12)*+{\RR^{(k)}\!\!\!\!\!}="4";
	{\ar@{^{(}->} "1"; "2"};
	{\ar@{->>} "1"; "3"};
	{\ar@{^{(}->} "3"; "4"};
	{\ar@{->}^{\varepsilon_{i}} "2"; "4"};
	\end{xy}
	\ \ \ \ 
	\end{aligned}
	\end{equation}
\end{lemma}

\section{Polyhedral geometry over Hahn embeddings}\label{section: polyhedral geometry over Hahn embeddings}

\subsection{$\pmb{\Gamma}$-Rational polyhedra over $\pmb{\RR^{(k)}}$}
\label{subsection: rational polyhedra}
Fix a totally ordered abelian group $\Gamma$ with Hahn embedding $\Gamma\hookrightarrow\RR^{(k)}$, and let $M$ be a lattice with dual $N:=\Hom_{\ZZ}(M,\ZZ)$. We can form the free abelian group	
	$$
	N\otimes_{\ZZ}\RR^{k}\ \cong\ \Hom_{\ZZ}(M,\RR^{k}).
	$$
Identifying $\RR^k$ with the group underlying $\RR^{(k)}$, the group $N\otimes_{\ZZ}\RR^{k}$ inherits a left $\mathcal E_k$ action. Let $N_{\RR^{(k)}}$ denote $N\otimes_{\ZZ}\RR^{k}$ equipped with this left $\mathcal{E}_{k}$-action. Observe that the tower (\ref{equation: convex subgroups of R^(k)}) of convex subgroups of $\RR^{(k)}$ induces a tower of abelian subgroups
	$$
	\{0\}\hookrightarrow N_{\RR}\hookrightarrow N_{\RR^{(2)}}\hookrightarrow\cdots\hookrightarrow N_{\RR^{(k)}},
	$$
with the property that for each $0\le j\le k$, the subsemiring $\mathcal{E}_{j}$ maps $N_{\RR^{(k)}}$ into the subgroup $N_{\RR^{(j)}}\subset N_{\RR^{(k)}}$. Furthermore, one has a canonical pairing
	\begin{equation}\label{equation: M pairing}
	\langle-,-\rangle:M\times N_{\RR^{(k)}}\lra\RR^{(k)}.
	\end{equation}
A {\em $\Gamma$-rational hyperplane}, respectively {\em $\Gamma$-rational halfspace}, is any subset of $N_{\RR^{(k)}}$ of the form
	$$
	\begin{array}{r}
	H^{0}
	\ =\ 
	\big\{v\in N_{\RR^{(k)}}:\langle u,v\rangle=\gamma\big\},\ \!\!
	\\[10pt]
	\mbox{respectively}\ \ 
	H^{\tinyge0}
	\ =\ 
	\big\{v\in N_{\RR^{(k)}}:\langle u,v\rangle\ge\gamma\big\},
	\end{array}
	\ \ \ \ \ \ \ \ \ \ \ \ \ \ \ \ \ \ \ \ \ \ \ 
	$$
for a fixed $u\in M$ and $\gamma\in\Gamma$. We sometimes write $H^{0}_{(u,\gamma)}$ and $H^{\ge0}_{(u,\gamma)}$ when we want to make the pair $(u,\gamma)$ explicit. When $\gamma=0$, we refer to $H^{0}$ and $H^{\tinyge0}$ as a halfspace and hyperplane {\em through the origin}. A {\em $\Gamma$-rational polyhedron} in $N_{\RR^{(k)}}$ is any subset $P\subset N_{\RR^{(k)}}$ that arises as the intersection of finitely many $\Gamma$-admissible halfspaces in $N_{\RR^{(k)}}$. If $P$ is a $\Gamma$-rational polyhedron in $N_{\RR^{(k)}}$ of the form $P=H^{\tinyge0}_{1}\cap\cdots\cap H^{\tinyge0}_{m}$, then a (non-empty) {\em face} of $P$ is the intersection obtained upon replacing any subset of the halfspaces $H^{\tinyge0}_{\ell}$ with their corresponding hyperplanes $H^{0}_{\ell}$. Every face of $P$ is itself a $\Gamma$-rational polyhedron in $N_{\RR^{(k)}}$.
	
	A {\em $\Gamma$-rational polyhedral complex} $\mathcal{P}$ in $N_{\RR^{(k)}}$ is any finite collection of $\Gamma$-rational polyhedra in $N_{\RR^{(k)}}$ satisfying the following two conditions:
	\begin{itemize}
	\item[\textbf{ (i)}]
	For each $P\in\mathcal{P}$, every face of $P$ is in $\mathcal{P}$;
	\item[\textbf{ (ii)}]\vskip .2cm
	For any $P_{1},P_{2}\in\mathcal{P}$, the intersection $P_{1}\cap P_{2}$ is a $\Gamma$-rational polyhedron in $\mathcal{P}$.
	\end{itemize}
	
	If $P$ is a $\Gamma$-rational polyhedron in $N_{\RR^{(k)}}$, then a {\em flag of faces  in $P$} is any tower of strict inclusions of (non-empty) faces
	$
	P_{0}\subsetneq\ P_{1}\subsetneq\cdots\subsetneq\ P_{m}=P
	$.
We refer to the nonnegative integer $m$ as the {\em rank} of the flag. The {\em dimension} of $P$ is the maximum rank of a flag of faces in $P$. A $0$-dimensional face of $P$ is called a {\em vertex}. A $1$-dimensional face is called an {\em edge}.

\begin{remark}
{\bf Largest linear subspaces and pointed quotients.}
	If we work over the {\em trivial Hahn embedding} $\{0\}\hookrightarrow\RR^{(k)}$, then our combinatorial geometry becomes that of $\{0\}$-rational polyhedra and their complexes. A {\em linear subspace} of $N_{\RR^{(k)}}$ is any finite intersection of $\{0\}$-rational hyperplanes in $N_{\RR^{(k)}}$. Because each $\{0\}$-rational hyperplane takes its complete determination from the vector $u\in M$ to which it is dual, each linear subspace $V\subset N_{\RR^{(k)}}$ is of the form
	$$
	V
	\ =\ 
	\big\{v\in N_{\RR^{(k)}}:\langle u,v\rangle=0\ \mbox{for all}\ u\in V^{\perp}\big\}
	$$
for a unique $\ZZ$-linear subspace $V^{\perp}\subset M$. Clearly $V$ is a subgroup of the abelian group $N_{\RR^{(k)}}$. Define $N_{\RR^{(k)}}/V$ to be the quotient of $N_{\RR^{(k)}}$ by $V$ in the category of abelian groups. If we let $V_{\ZZ}$ denote the $\ZZ$-linear dual of $V^{\perp}$ inside $N$, then there is a canonical isomorphism of abelian groups
	\begin{equation}\label{equation: canonical N/V isomorphism}
	N_{\RR^{(k)}}/V
	\ \ \cong\ \ 
	\big(N/V_{\ZZ}\big)_{\RR^{(k)}}.
	\end{equation}
	
	If $P$ is a $\Gamma$-rational polyhedron in $N_{\RR^{(k)}}$, then $\Gamma$-translates of $P$ will contain varying linear subspaces. Because the faces of $P$ are ordered by inclusion, with a unique maximal face, there exists a unique {\em largest linear subspace} contained in at least one of these translates. We say that $P$ is {\em pointed} if its largest linear subspace is $V=\{0\}$. If $V$ is the largest linear subspace of $P$, then the image of $P$ in $N_{\RR^{(k)}}/V$ is a pointed $\Gamma$-rational polyhedron.
\end{remark}

\begin{remark}\label{rmk: fans}
{\bf Fans.}
	 A {\em cone} in $N_{\RR^{(k)}}$ is a pointed $\{0\}$-rational polyhedron $\sigma\subset N_{\RR^{(k)}}$. A {\em fan} in $N_{\RR^{(k)}}$ is any $\{0\}$-rational polyhedral complex consisting entirely of cones. A fan $\Sigma$ in $N_{\RR^{(k)}}$ can be completely recovered from the collection $\{S_{\sigma}\}_{\sigma\in\Sigma}$ of dual semigroups
	$$
	S_{\sigma}\ :=\ \big\{u\in M:\langle u,v\rangle\ge0\ \mbox{for all}\ v\in\sigma\big\}.
	$$
Explicitly, we can take the cones in $N_{\RR^{(k)}}$ to be $\Hom(S_\sigma,\RR^{(k)}_{\geq 0})$ and glue along faces. In this way, a fan in $N_{\RR^{(k)}}$ contains no more information than a fan in $N_{\RR}$. One may informally think of this as a ``base change'' of the fan along the order preserving projection $\RR^{(k)}\to \RR$. 

	The {\em star} of a polyhedron $P$ in a polyhedral complex $\mathcal P$ is defined in direct analogy with the standard definition in the rank-$1$ case, as the collection of cones of unbounded directions. These cones glue to form a fan whose cones are indexed by the cells of $\mathcal P$ that contain $P$.
\end{remark}

\subsection{$\pmb{\Gamma}$-Admissible fans in $\pmb{N_{\RR^{(k)}}\!\times\!\mathcal{E}}$} 
The pairings (\ref{equation: M pairing}) and (\ref{equation: gamma pairing}) induce a single pairing
	$$
	\ \ \ 
	\begin{array}{rcl}
	\big(M\!\times\!\Gamma\big)\times\big(N_{\RR^{(k)}}\!\times\!\mathcal{E}\big)
	&
	\xrightarrow{\ \ \ \ \ }
	&
	\RR^{(k)}
	\\[10pt]
	\ \ \ \!(u,\gamma)\ \ \ \ \&\ \ \ \ \ \ (v,\ \!\varphi)\ \ \ \ \ \ 
	&
	\!\!\!\!\!\!\!\!\!\!\!\!\!\!\!\!\!\!\longmapsto
	&
	\!\!\!\!\!\!\!\!\!\!\langle u,\ \!v\rangle+\varphi(\gamma).
	\end{array}
	$$
When $k=1$, we have $N_{\RR^{(1)}}\!\times\mathcal{E}=N_{\RR\!}\times\RR_{\tinyge0}$, and the above pairing becomes the pairing that Gubler employs throughout \cite{Gub13}.
	
	A {\em $\Gamma$-admissible halfspace} in $N_{\RR^{(k)}}\!\times\!\mathcal{E}$ is any subset of the form
	$$
	H^{\tinyge0}
	\ =\ 
	\big\{(v,\varphi)\in N_{\RR^{(k)}}\!\times_{\!}\mathcal{E}:\langle u,\ \!v\rangle+\varphi(\gamma)\ge0\big\}
	$$
for a fixed pair $(u,\gamma)\in M\times\Gamma$. A {\em $\Gamma$-admissible cone} in $N_{\RR^{(k)}}\!\times\!\mathcal{E}$ is any subset $\sigma\subset N_{\RR^{(k)}}\!\times\!\mathcal{E}$ that is a finite intersection of $\Gamma$-admissible halfspaces in $N_{\RR^{(k)}}$, such that $\sigma$ does not contain any $1$-dimensional linear subspace of $N_{\RR^{(k)}}\!\times\!\{0\}\cong N_{\RR^{(k)}}$. If $\sigma$ is a $\Gamma$-admissible cone in $N_{\RR^{(k)}}\!\times\!\mathcal{E}$ of the form $\sigma=H^{\tinyge0}_{1}\cap\cdots\cap H^{\tinyge0}_{m}$, then a {\em face} of $\sigma$ is any one of the $\Gamma$-admissible cones that we obtain upon replacing any subset of the $\Gamma$-admissible halfspaces $H^{\tinyge0}_{\ell}$ in this intersection with their corresponding $\Gamma$-admissible hyperplanes $H^{0}_{\ell}$. A {\em $\Gamma$-admissible fan} in $N_{\RR^{(k)}}\!\times\!\mathcal{E}$ is any finite collection $\Sigma$ of $\Gamma$-admissible cones $\sigma\subset N_{\RR^{(k)}}\!\times\!\mathcal{E}$ satisfying the conditions:
	\begin{itemize}
	\item[\textbf{ (i)}]
	For each $\sigma\in\Sigma$, every face of $\sigma$ is in $\Sigma$;
	\item[\textbf{ (ii)}]\vskip .2cm
	For any $\sigma_{1},\sigma_{2}\in\Sigma$, the intersection $\sigma_{1}\cap \sigma_{2}$ is a $\Gamma$-admissible cone in $\Sigma$.
	\end{itemize}

\subsection{Recession complexes}
	Let $\sigma$ be a $\Gamma$-admissible cone in $N_{\RR^{(k)}}\!\times\!\mathcal{E}$. Then for each $0\le i\le n$, the {\em $i^{\mathrm{th}}$-recession polyhedron of $\sigma$}, denoted $\mathrm{rec}_{i}(\sigma)\subset N_{\RR^{(k)}}$, is the image of the set
	$$
	\sigma\cap\big(N_{\RR^{(k)}}\!\times\!\{\varepsilon_{i}\}\big)
	$$
under the projection $N_{\RR^{(k)}}\!\times\!\{\varepsilon_{i}\}\twoheadrightarrow N_{\RR^{(k)}}$. If $\Sigma$ is a $\Gamma$-admissible fan in $N_{\RR^{(k)}}\times\mathcal{E}_{0}$, then for each $0\le i\le  n$, the {\em $i^{\mathrm{th}}$-recession complex of $\Sigma$}, denoted $\mathrm{rec}_{i}(\Sigma)$, is the collection
	$$
	\mathrm{rec}_{i}(\Sigma)\ :=\ \big\{\mathrm{rec}_{i}(\sigma)\big\}_{\sigma\in\Sigma}.
	$$

\begin{proposition}\label{proposition: recession complex behave well}
Let $\Sigma$ be a $\Gamma$-admissible fan in $N_{\RR^{(k)}}\!\times\!\mathcal{E}$. For each $0\le i\le n$, let $\Gamma/\Delta_{i}\hookrightarrow\RR^{(k)}$ be the embedding (\ref{equation: important commutative square}). Then:
	\begin{itemize}
	\item[{\bf (i)}]
	For each $0\le i\le n$, the $i^{\mathrm{th}}$-recession complex $\mathrm{rec}_{i}(\Sigma)$ is a $\Gamma/\Delta_{i}$-rational polyhedral complex in $N_{\RR^{(k)}}$.
	\item[{\bf (ii)}]\vskip .2cm
	The $n^{\mathrm{th}}$-recession complex $\mathrm{rec}_{n}(\Sigma)$ is a fan in $N_{\RR^{(k)}}$.
	\end{itemize}
\end{proposition}

Note that the second conclusion is a special case of the first, since $\{0\}$-rational polyhedral complexes are fans.

\begin{proof}
If $H^{\ge0}$ is the $\Gamma$-admissible halfspace determined by the pair $(u,\gamma)\in M\times\Gamma$, then the image of $H^{\ge0}\cap\big(N_{\RR^{(k)}}\!\times\!\{\varepsilon_{i}\}\big)$ under the projection $N_{\RR^{(k)}}\!\times\!\{\varepsilon_{i}\}\twoheadrightarrow N_{\RR^{(k)}}$ is the set
	$$
	\big\{v\in N_{\RR^{(k)}}:\langle u,\ \!v\rangle\ge-\varepsilon_{i}(\gamma)\big\}.
	$$
Clearly $-\varepsilon_{i}(\gamma)\in\varepsilon_{i}(\Gamma)$, and thus part {\bf (i)} follows from the fact that $\varepsilon_{i}(\Gamma)$ is the image of $\Gamma/\Delta_{i}$ under the embedding $\Gamma/\Delta_{i}\hookrightarrow\RR^{(k)}$.
	
	When $i=n$, we have $\Gamma/\Delta_{n}=\{0\}$, and $\mathrm{rec}_{n}(\Sigma)$ is $\{0\}$-rational. If there exists a $\Gamma$-admissible cone $\sigma \in\Sigma$ such that $\mathrm{rec}_{n}(\sigma)$ contains a $\{0\}$-admissible line $L$ through the origin in $N_{\RR^{(k)}}$, then $L\times\{\varepsilon_{n}\}$ is a $\Gamma$-admissible line through the origin in $N_{\RR^{(k)}}\!\times\!\mathcal{E}$ contained in $\sigma$. This contradicts the fact that $\sigma$ is a $\Gamma$-admissible cone.
\end{proof}

\section{Models associated to polyhedral complexes}\label{section: models associated to polyhedral complexes}

Fix a field $K$ with valuation $v:K^{\times}\to\RR^{(k)}$, and define $\Gamma:=v(\Gamma)\subset\RR^{(k)}$. The resulting inclusion $\Gamma\hookrightarrow\RR^{(k)}$ is a Hahn embedding. Let (\ref{equation: convex subgroups in value group}) be the maximal tower of convex subgroups in $\Gamma$, so that $n=\mathrm{rank}_{\ \!}\Gamma$, and let $R_{i}$ and $\widetilde{K}_{i}$ be the corresponding intermediate valuation subrings and residue fields of $K$ as described in Section~\ref{subsection: finite rank valuation rings}.
	
	We now connect the combinatorial $\RR^{(k)}$-geometry of Sections \ref{section: structure of value groups} and \ref{section: polyhedral geometry over Hahn embeddings} to the geometry of degenerations of toric varities.

\subsection{Models associated to polyhedra}\label{sec: polyhedra-models}
	Given a $\Gamma$-rational polyhedron $P$ in $N_{\RR^{(k)}}$, its associated {\em $P$-tilted algebra} is the $R$-algebra
	$$
	R[M]^{P}
	\ :=\ 
	\Big\{\ \sum_{u\in M}a_{u}\chi^{u}\in K[M]\ :\ \langle u,v\rangle+\nu(a_{u})\ge0\ \mbox{for all}\ v\in P,\ u\in M\ \Big\}.
	$$
The {\em polyhedral model associated to $P$} is the affine $R$-scheme
	$$
	\mathscr{U}\!(P)
	\ :=\ 
	\spec_{\ \!}R[M]^P.
	$$
If $\sigma\subset N_{\RR^{(k)}}\!\times\!\mathcal{E}$ is a $\Gamma$-admissible cone, then we let $\mathscr{U}\!(\sigma):=\mathscr{U}\!\big(\mathrm{rec}_{0}(\sigma)\big)$ denote the polyhedral model associated to the $0^{\mathrm{th}}$ recession polyhedron of $\sigma$.
	
\begin{lemma}\label{lemma: passage to pointed polyhedron}
If $V$ is the largest linear subspace of $P$, let $P/V$ denote the image of $P$ under the map $N_{\RR^{(k)}}\lra\!\!\!\!\rightarrow N_{\RR^{(k)}}/V$. Then we have an isomorphism of $R$-algebras $R[M]^{P}\cong R[V^{\perp}]^{P/V}$.
\end{lemma}
\begin{proof}
Suppose that $a_{u}\chi^{u}$ is a monomial in $R[M]^{P}$, i.e., that $\langle u,v\rangle+\nu(a_{u})\ge0$ for all $v\in P$. If $u\notin V^{\perp}$, then there is a homomorphism $v_{1}:M\to\ZZ$ such that $v_{1}(u)\ne 0$, and $v_{1}(u')=0$ for all $u'\in V^{\perp}$. Let $v:M\to\RR^{(k)}$ be the map that returns $\langle u',v\rangle=\big(v_{1}(u'),\ 0,\ \dots,\ 0\big)\in\RR^{(k)}$ at each $u'\in M$. Then $v\in V$, and $\langle u,v\rangle$ is in the largest Archimedean class of $\RR^{(k)}$. This implies that there exists $m\in\ZZ$ so that $\langle u, m\ \!v\rangle\ge-\nu(a_{u})$, contradicting our choice of $u$ and $v$. Thus $u\in V^{\perp}$. From the isomorphism (\ref{equation: canonical N/V isomorphism}) we see that $a_{u}\chi^{u}\in R[V^{\perp}]^{P/V}$. The converse follows immediately from (\ref{equation: canonical N/V isomorphism}) and the definition of $P/V$.
\end{proof}
	
	\begin{proposition}\label{prop: flat, integral domain, integrally closed}
The R-algebra $R[M]^P$ is flat, is an integral domain, and is integrally closed in its fraction field.
\end{proposition}
\begin{proof}
	Since $R[M]^P$ is a subring of $K[M]$, it is an integral domain. To see that $R[M]^P$ is flat, it suffices to check  that for any finitely generated ideal $I\subset R$, the map $I\otimes_R R[M]^P\to R[M]^P$ is injective \cite[Proposition 2.19(iv)]{AM69}. But $R$ is a valuation ring, so every finitely generated ideal is principal. Injectivity is for such ideals follows from the fact that $K[M]$ has no torsion elements. To verify that $R[M]^P$ is integrally closed in its fraction field, first note that by Lemma \ref{lemma: passage to pointed polyhedron}, we may assume that $P$ is pointed. This implies that every $\RR^{(k)}$-valued $\Gamma$-affine functional on $N_{\RR^{(k)}}$ takes its minimum at a vertex of $P$. We deduce that the algebra $R[M]^P$ is the intersection of $R[M]^{\{v\}}$ over vertices $v$ of $P$. It suffices now to show the claim for $R[M]^{\{v\}}$. By modifying the standard arguments for monomial valuations on polynomial rings~\cite[Proposition 2.1.2]{Ked}, one may verify that the assignment $\nu_{\{v\}}:R[M]^{\{v\}}\to\RR^{(k)}$ taking $a\chi^{u}\mapsto\langle u,v\rangle+\nu(a)$ is a valuation. The claim then follows by the usual argument for integral closure of valuation rings~\cite[Proposition 5.18(iii)]{AM69}.
\end{proof}

\begin{proposition}\label{prop: finite-pres}
If $K$ is algebraically closed, then $R[M]^P$ is of finite presentation over $R$.
\end{proposition}
\begin{proof}
We follow~\cite[Proposition 4.11]{BPR}. By Lemma \ref{lemma: passage to pointed polyhedron}, we may assume that $P$ is pointed. Because every finitely generated flat algebra over an integral domain must be of finite presentation~\cite[Corollary 3.4.7]{RG71}, Proposition \ref{prop: flat, integral domain, integrally closed} reduces the proof to a check that $R[M]^{P}$ finitely generated. Let $v_1,\ldots, v_r$ be the vertices of $P$. Write $\sigma_i$ for the star at $v_i$. The union of the dual cones $\sigma_i^\vee$ is $\sigma_P^\vee$. It is easy to see that $R[M]^P$ is generated by subrings $R[M]^P\cap K[S_{\sigma_j}]$, and thus it suffices, for each $j$, to find a finite set that generates $R[M]^P\cap K[S_{\sigma_j}]$.

Fix an index $j$. The monoid $S_{\sigma_j}$ is finitely generated, say by $u_1,\ldots, u_s$. Consider the element $\langle u_i,v_j\rangle$. Since $\Gamma$ is divisible, the vertex $v_j$ is a to be a $\Gamma$-rational point of $N_{\RR^{(k)}}$, and thus we may choose scalars $a_i\in K^\times$ such that $\nu(a_{i})+\langle u_i, v_j\rangle$ is zero. The set $\{a_i \chi^{u_i}\}$ now generates $A_j$ over $R$.
\end{proof}

\begin{remark}\label{remark: the R-torus, etc}
Let $\slantbox[.6]{$\mathcal{T}\ $}$ denote the $R$-torus $\spec_{\ \!}R[M]$. Then for each $\Gamma$-rational polyhedron $P$ in $N_{\RR^{(k)}}$, the diagonal $R$-morphism $R[M]^{P}\!\to R[M]\otimes_{R}R[M]^{P}$ equips the polyhedral model $\mathscr{U}\!(P)$ with a natural $\slantbox[.6]{$\mathcal{T}\ $}$-action. let $T$ denote the generic torus $T:=\slantbox[.6]{$\mathcal{T}\ $}_{\!K}=\mathrm{Spec}_{\ \!}K[M]$.
\end{remark}

\begin{lemma}\label{prop: faces-glue}
If $Q$ is a face of a $\Gamma$-rational polyhedron $P$ in $N_{\RR^{(k)}}$, then the inclusion $Q\hookrightarrow P$ induces a $\slantbox[.6]{$\mathcal{T}\ $}$-equivariant inclusion $\mathscr{U}\!(Q)\hookrightarrow\mathscr{U}\!(P)$ that identifies $\mathscr{U}\!(Q)$ with a distinguished affine open inside $\mathscr{U}\!(P)$.
\end{lemma}
\begin{proof}
The proof is similar to the proof of the analogous result in toric $K$-varieties. See for instance~\cite[Proposition 6.12]{Gub13}.
\end{proof}

\begin{definition}
	The {\em model} of a $\Gamma\!$-rational polyhedral complex $\mathcal{P} \subset N_{\RR^{(k)}}\!$ is the $R$-scheme
	$$
	\mathscr{Y}\!(\mathcal{P})\ :=\ {\varinjlim}_{P\in\mathcal{P}}\mathscr{U}\!(P),
	$$
glued via Lemma \ref{prop: faces-glue}. If $\Sigma$ is a $\Gamma$-admissible fan in $N_{\RR^{(k)}}\!\times\!\mathcal{E}$, then we let $\mathscr{Y}\!(\Sigma)$ denote the model $\mathscr{Y}\!\big(\mathrm{rec}_{0}(\Sigma)\big)$ of $\Sigma$'s $0^{\mathrm{th}}$-recession complex.
\end{definition}

\begin{example}\label{example: 2-stage degeneration}
{\bf A $\pmb{2}$-stage degeneration of $\pmb{\PP^1}$.}
Let $K$ be the Hahn series field $\CC\llbracket \RR^{(2)}\rrbracket$, and let $N$ be a rank-$1$ lattice $N\cong\ZZ$. Consider the following polyhedral decomposition $\mathcal P$ of  $\RR^{(2)}$:
\begin{enumerate}[(1)]
\item The vertices are given by $(0,-1), (0,1)$, and $(1,0)$. 
\item The $1$-dimensional polyhedral are the intervals $(-\infty,(0,-1)]$, $[(0,-1),(0,1)]$, $[(0,1),(1,0)]$ and $[(1,0),\infty)$. 
\end{enumerate}
	This decomposition is the $0^{\mathrm{th}}$-recession complex $\mathrm{rec}_0(\Sigma)$ of a $\ZZ^{(2)}$-rational fan $\Sigma$ in $N_{\RR^{(2)}}\!\times\!\mathcal{E}$ given by the cone over $\mathcal P$ (see Figure \ref{figure: recession complexes} ).
	\begin{figure}[h!]    
	$$
	\scalebox{.85}{$
	\begin{xy}
	(0,17.5)*+{\includegraphics[scale=.45]{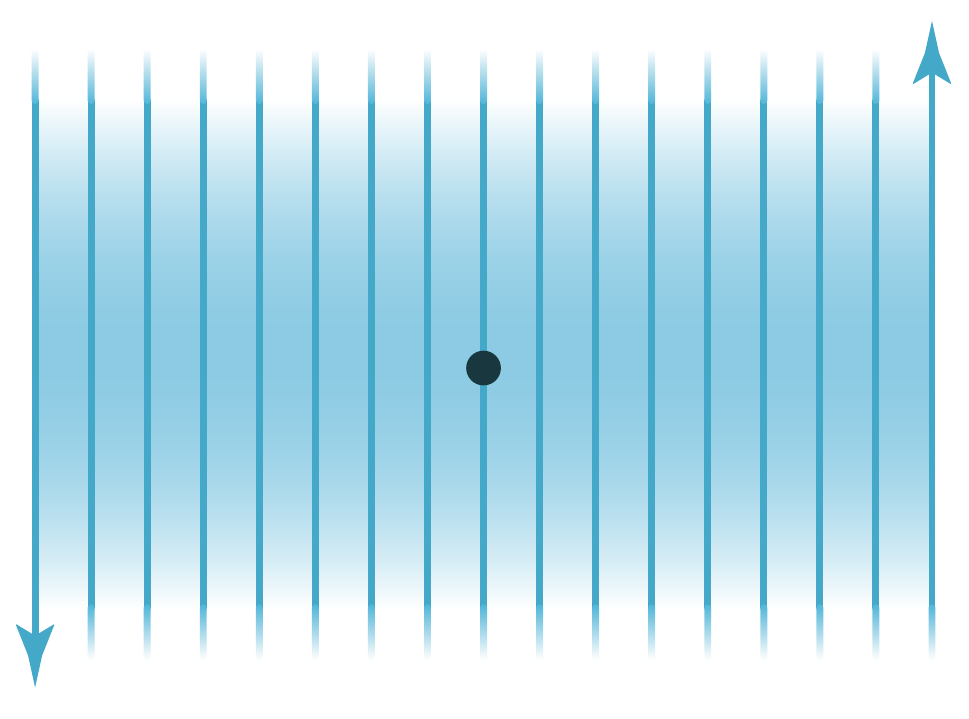}};
	(0,-1)*+{\mathrm{rec}_{2}(\Sigma)};
	\end{xy}
	\ \ \ \ \ \ \ \ \ \ \ 
	\begin{xy}
	(0,17.5)*+{\includegraphics[scale=.45]{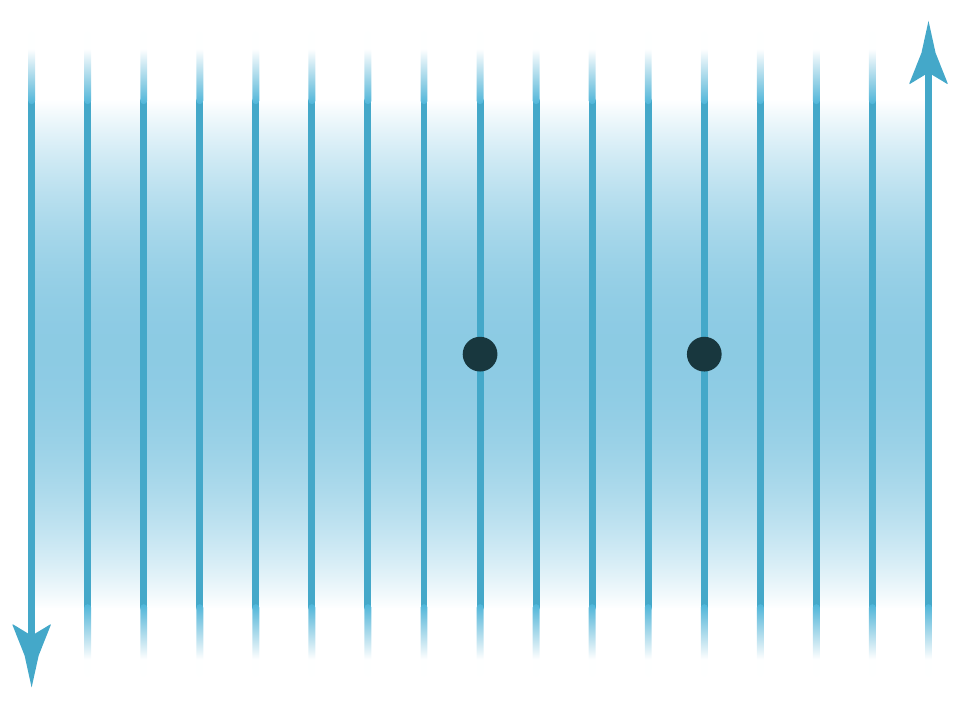}};
	(1.45,15.75)*+{\mbox{\smaller a}};
	(11.75,15.75)*+{\mbox{\smaller b}};
	(0,-1)*+{\mathrm{rec}_{1}(\Sigma)};
	\end{xy}
	\ \ \ \ \ \ \ \ \ \ \ 
	\begin{xy}
	(0,17.5)*+{\includegraphics[scale=.45]{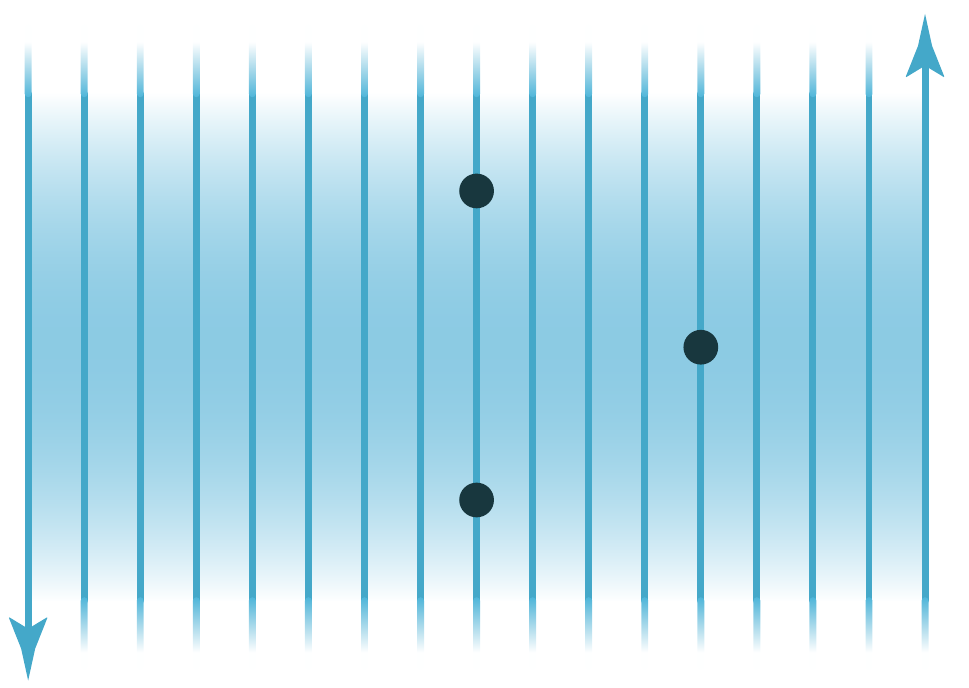}};
	(1.45,8.5)*+{\mbox{\smaller c}};
	(1.45,26.75)*+{\mbox{\smaller d}};
	(11.75,15.75)*+{\mbox{\smaller e}};
	(0,-1)*+{\mathrm{rec}_{0}(\Sigma)};
	\end{xy}
	$}
	$$\vskip -.25cm
	\caption{The three recession fans of the $\ZZ^{(2)}$-rational fan $\Sigma$ in Example \ref{example: 2-stage degeneration}.}
	\label{figure: recession complexes}    
	\end{figure}
The special fiber of the model $\mathscr Y(\Sigma)$ consists of $3$ $\PP^1$'s glued in a chain (see Figure \ref{figure: 2-stage degeneration}). As in Figure \ref{figure: recession complexes}, the first recession complex $\mathrm{rec}_1(\Sigma)$ has two vertices $(0,-1)$ and $(0,1)$, and correspondingly, the intermediate fiber consist of two $\PP^1$'s glued at a single point, whereas the second recession complex $\mathrm{rec}_2(\Sigma)$ is the fan of $\PP^1$ over the value group $\RR^{(2)}$.
	\begin{figure}[h!]    
	$$
	\scalebox{.85}{$
	\begin{xy}
	(0,-.5)*+{\includegraphics[scale=.65]{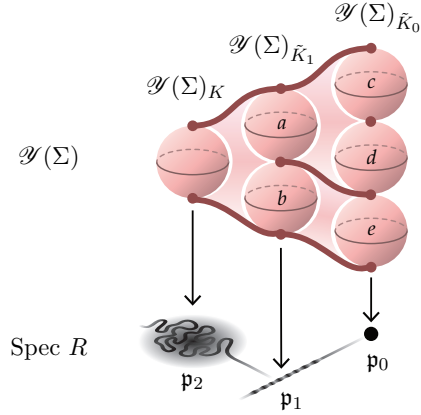}};
	(-35,-20)*+{\Spec_{\ \!}R}="2";
	(-35,10)*+{\mathscr{Y}(\Sigma)}="1";
	(1.5,15.25)*+{\mbox{\smaller\smaller a}};
	(1.5,3.75)*+{\mbox{\smaller\smaller b}};
	(15.5,21.75)*+{\mbox{\smaller\smaller c}};
	(15.5,10.25)*+{\mbox{\smaller\smaller d}};
	(15.5,-1.25)*+{\mbox{\smaller\smaller e}};
	(-12.5,-25)*+{\mathfrak{p}_{2}};
	(3,-28)*+{\mathfrak{p}_{1}};
	(16.5,-22)*+{\mathfrak{p}_{0}};
	(-13.5,21)*+{\mathscr{Y}(\Sigma)_{K}};
	(0,27)*+{\mathscr{Y}(\Sigma)_{\tilde{K}_{1}}};
	(16.5,32)*+{\mathscr{Y}(\Sigma)_{\tilde{K}_{0}}};
	\end{xy}
	$}\ \ \ \ \ \ \ \ \ \ \ \ \ \ \ \ 
	$$\vskip -.25cm
	\caption{The model $\mathscr Y(\Sigma)$ from Example \ref{example: 2-stage degeneration}. Thick red lines represent horizontal, torus-invariant divisors in the intermediate special fibers.}
	\label{figure: 2-stage degeneration}
	\end{figure}
\end{example}

\begin{example}{\bf Toric degenerations from toric morphisms.}
A large class of interesting examples of multistage toric degenerations, in the spirit of those in~\cite{NS06}, may be formed by the following construction. Consider a separated, flat, equivariant morphism of toric varieties $Y(\Sigma)\to U_\sigma$, where $U_\sigma$ is an affine toric variety $\Spec(K[S_\sigma])$. Choose a flag $F_\bullet$ of faces
\[
\{0\}=\sigma_0\subset \sigma_1\subset \sigma_2\subset\cdots\subset \sigma_k,
\]
where $\dim \sigma_k = k$. This determines a descending flag of orbit closures 
\[
Z_0\subset Z_{k-1}\subset\cdots Z_{0} = U_\sigma.
\]
There is a sequence of characters $\chi_1,\chi_{2},\ldots, \chi_k$, such that $\chi_i$ cuts out $Z_i$ inside $Z_{i+1}$. This determines a valuation
\[
v_{F_\bullet}:K[U_\sigma]\to \ZZ^{(k)}\hookrightarrow \RR^{(k)},
\]
obtained by sequentially measuring order of vanishing of a function along the flag of torus orbits, see for instance~\cite[Section 2.1]{Ok96}. Note that any valuation of this form $v_{F_\bullet}$ is naturally a point of the Hahn analytic space $U_\sigma^\fH$, in the sense of~\cite{FR1}. Passing to the associated valuation ring $R$ and base changing the morphism $Y(\Sigma)\to U_\sigma$ we obtain a multistage toric degeneration
\[
\mathscr Y\to \spec(R). 
\]

	This degeneration can also be obtained using a polyhedral geometry construction. To see this, we observe that the point $v_{F_\bullet}$ determines a point of $p_{F_\bullet}\in \Hom(S_\sigma, \RR^{(k)})$. By Remark~\ref{rmk: fans}, any morphism of fans $\Sigma\to \sigma$ is equivalent to the data of the corresponding morphism of fans over $\RR^{(k)}$, which we continue to denote $\Sigma\to \sigma$. The fiber in $\Sigma$ over $p_{F_\bullet}$ inherits the structure of a $\ZZ^{(k)}$-rational polyhedral complex over $\RR^{(k)}$ in $N_{\RR^{(k)}}$. Translating this polyhedral complex, we obtain an admissible fan in the sense of the previous section and hence a toric degeneration. The toric degeneration coincides with $\mathscr Y$ by an explicit calculation of the corresponding monoid rings. 
\end{example}

\begin{proof}[{\textit Proof of part {\bf (iii)} of the Main Theorem:}]
It suffices to consider the affine case, i.e. to consider the case where $\mathcal P$ is a single polyhedron. Let $P$ be a polyhedron with recession cone $\sigma$ and let $R[M]^P$ be the associated tilted algebra. Consider the $K$-algebra $R[M]^P\otimes_R K$. Observe that the polyhedron $P$ is closed under the addition of points of its recession cone $\sigma$. It follows that any Laurent polynomial satisfying the given inequalities for all $v\in P$ must be supported in $\sigma^\vee$. We conclude $R[M]^P\otimes K\subset K[S_\sigma]$, so we must prove the reverse inclusion. Given a Laurent polynomial $\sum a_u \chi^u$, the minimum over $v\in P$ of the affine functional $\nu(a_u)+\langle u,v\rangle$ must attain a minimum at a vertex of $P$. We may find an element $\lambda\in K$ such that $\lambda\cdot (\sum a_u\chi^u)$ lies in $R[M]^P$, and the desired inclusion follows. 
\end{proof}

\begin{definition}\label{definition: weight function}
	For each $\Gamma$-rational polyhedron $P$ in $N_{\RR^{(k)}}$, the {\em weight function} associated to $P$ is the map $\nu_{P}:K\big[S_{\mathrm{rec}_{0}(P)}\big]\lra\RR^{(k)}$ taking
	$$
	 \sum_{u\in M} a_u\chi^u\ \longmapsto\  \inf_{w\in P}\big\{\nu(a_u)+\langle u,w\rangle\big\}.
	$$
\end{definition}

\begin{remark}
	We point out that $\nu_P$ is a natural object in the study of higher rank valued geometry -- it is a point of the Hahn analytic space associated to $K\big[S_{\mathrm{rec}_{0}(P)}\big]$ ~\cite{FR1}.
\end{remark}

\begin{lemma}\label{lem: v-P-lemma}
For regular functions $f,g\in K\big[S_{\mathrm{rec}_{0}(P)}\big]$ and scalar $a\in K$, the weight function $\nu_P$ satisfies the following properties.
\begin{enumerate}[(W1)]
\item The weight function is supmultiplicative: $\nu_P(f\cdot g) \geq \nu_P(f)+\nu_P(g)$.
\item The weight function is power-multiplicative: $\nu_P(f^m) = m\cdot \nu_P(f)$.
\item The weight function is linear for $K$-scalars: $\nu_P(a \cdot f) = \nu(a)+\nu_P(f)$.
\end{enumerate}
\end{lemma}
\begin{proof}
As before, we may quotient $N_{\RR^{(k)}}$ by the largest linear subspace contained in $\sigma_P$, and thus assume that $P$ is pointed. For pointed $P$, every $\RR^{(k)}$-valued $\Gamma$-affine functional on $N_{\RR^{(k)}}$ takes its minimum at a vertex of $P$. Thus, in Definition \ref{definition: weight function} above, we have 
	$$
	\inf_{w\in P}\big\{\nu(a_u)+\langle u,w\rangle\big\}
	\ \  =\ \ 
	\underset{w\ \textnormal{of}\ P}{\underset{\textnormal{vertices}}{\min}}\big\{\nu(a_u)+\langle u,w\rangle\big\}.
	$$
For each vertex $w$ of $P$, the standard argument shows that $\nu_w$ is a valuation on $K[S_{\mathrm{rec}_0(P)}]$, see for instance~\cite[Proposition 2.1.2]{Ked}. The claims (W1) through (W3) follow.
\end{proof}

\begin{theorem}\label{theorem: the big'un}
	There is a natural bijection between the vertices of the $(k-j_{i})^{\mathrm{th}}$-recession polyhedron $\mathrm{rec}_{k-j_{i}}(P)$ and the components of the $i^{\mathrm{th}}$-intermediate fiber of the $R$-scheme $\mathscr U_P$. The reduced induced structure of the component corresponding to a vertex $w\in\mathrm{rec}_{k-j_{i}}(P)$ is given by the vanishing of all functions $f$ in the set
	$$
	\big\{f\in R_{i}\!\otimes_{R}\!R[M]^P: pr^{k}_{k-j_{i}}\big(\nu_w(f)\big)>0\big\}.
	$$
The reduced component corresponding to $w$ is equivariantly isomorphic to the toric $\overline{K}_{i}$-variety given by the star of $w$ in $\mathrm{rec}_{k-j_{i}}(P)$.
	
	If $K$ is algebraically closed, then all intermediate fibers are reduced. 
\end{theorem}

\begin{proof}
The $i^{\mathrm{th}}$ intermediate fiber of the polyhedral model $\mathscr U_P$ is isomorphic to the special fiber of the $R_{\fp_i}$-model $\mathscr U_P\otimes R_{\fp_i}$. In turn, $R_{\fp_i}$ is the ring of nonnegative elements for the composite valuation
	$$
	K^\times\ \lra\!\!\!\!\rightarrow\ \Gamma\ \lra\!\!\!\!\rightarrow\ \Gamma\big/\Delta_{i}.
	$$
It is also the ring of nonnegative elements for the composition
\[
K^\times\ \lra\!\!\!\!\rightarrow\ \Gamma \hookrightarrow \RR^{(k)}\xrightarrow{{\rm pr}^k_{k-j_i}} \RR^{(k-{j_i})}. 
\]	
For $i = 0$, the recession polyhedron is the recession cone and the of $P$, and the fiber is the generic fiber of the model. In this case the claim is clear, so we assume $j\geq 1$. From Lemma~\ref{lem: embeddedquotients}, the vertices of recession polyhedron $\mathrm{rec}_{k-j_{i}}(P)$ are in bijection with the vertices of the image of $P$ under the natural projection
	$$
	\pi^k_{k-j_{i}}\ \!:\ N_{\RR^{(k)}}\ \lra\!\!\!\!\rightarrow\ N_{\RR^{(k-j_{i})}}
	$$
induced by the continuous projection $\mathrm{pr}^{k}_{j}:{\RR^{(k)}}\lra\!\!\!\!\rightarrow\RR^{(k-j)}$. To see the bijection between the vertices of $\mathrm{rec}_{k-j_{i}}(P)$ and components of the $j$th intermediate fiber, it now suffices to construct a bijection between the components of this fiber and the vertices of $\pi^k_{k-j_{i}}(P)$.

The fiber of the model $\mathscr U_P\otimes R_{\fp_i}$ above the closed point of $R_{\fp_i}$ is cut out by $I = \fp_i R_{\fp_i}[M]^P$ in $R_{\fp_i}[M]$. From Lemma~\ref{lem: v-P-lemma}, the weight function $\nu_P$ is power multiplicative, so the radical ideal $\sqrt{I}$ is contained in $J = \big\{f\in R_{i}\!\otimes_{R}\!R[M]^P: pr^{k}_{k-j_{i}}\big(\nu_w(f)\big)>0\big\}$. To see the reverse containment, first observe that since $J$ is $M$-graded, it suffices to check containment for every (possibly non-pure) monomial $\alpha\chi^u\in J$. Since $K$ has elements with valuation in the $j_i^{\mathrm{th}}$ Archimedean class, we may choose a scalar $a\in \fp_i$ such that $\nu(a)<L\cdot \nu_{\textrm{rec}_{k-j_i}(P)}(f)$ for some integer $L>0$. This proves that $f^L$ is contained in $I$, so $f\in \sqrt{I}$ as desired. If $K$ is algebraically closed, then we may choose $L = 1$ by taking $L^{\mathrm{th}}$ roots, so reducedness follows in this case. 

Finally, the isomorphism type of the component is given by the star of the corresponding vertex of $\mathrm{rec}_{k-j_{i}}(P)$. To see this, observe that the stars of $u$ and $\pi^k_{k-j_{i}}(u)$ yield isomorphic toric varieties. The result follows. 
\end{proof}

\begin{proof}[\textsc{Proof of parts {\bf (i)} and {\bf (ii)} of the Main Theorem:}]
Given a $\Gamma$-admissible fan $\Sigma$, the degeneration $\mathscr Y(\Sigma)$ is obtained by gluing the models associated to individual cones. For each of these cones, part (i) of the main theorem is an immediate consequence of Theorem~\ref{theorem: the big'un} above. The global version follows by gluing, using Lemma~\ref{prop: faces-glue}. The fact that the global model is separated follows from~\cite[Lemma 7.8]{Gub13}.  Part (ii) follows from Theorem~\ref{theorem: the big'un}.
\end{proof}
	
\bibliographystyle{siam}
\bibliography{GublerModels.bib}

\end{document}